\documentclass[10pt]{amsart}
\usepackage{graphics,verbatim,color,amsthm}
\usepackage{enumerate}
\usepackage[all]{xy}

\usepackage{setspace}

\theoremstyle{plain}
\newtheorem{theorem}{Theorem}

\newtheorem{prop}{Proposition}

\newtheorem{lemma}{Lemma}  
\newtheorem{cor}{Corollary} 
\newtheorem{definition}{Definition}

\newcommand{\R }{\mathbb R}

 \newcommand{\ga}{\gamma}
 \newcommand{\ep}{\epsilon}

\newcommand{\E }{\mathbb E}

\newcommand{\C }{\bf   C}

\DeclareMathOperator{\trace}{trace}
\DeclareMathOperator{\diag}{diag}

\def\g{\gamma}
\def\fr{\frac}
\def\into{\rightarrow}

\def\m#1{\begin{bmatrix}#1\end{bmatrix}}

\def\m#1{\begin{bmatrix}#1\end{bmatrix}}

\def\a{\alpha}
\def\b{\beta}
\def\g{\gamma}
\def\w{\omega}
\def\l{\lambda}
\def\e{\epsilon}
\def\d{\delta}

\newcommand{\fui}{\varphi}
  \newcommand{\af}{\alpha}
  \newcommand{\be}{\beta}
    
\newcommand{\mS}{\mathbb S}
\newcommand{\graf}{\operatorname{graph}}

\def\nablam{{\nabla_{\! m}}}
\def\massmetric#1#2{{\langle #1 , #2 \rangle_m}}

\begin{document}

\title{Free time minimizers for the planar three-body problem}

\title[Free time minimizers]{Free time minimizers for the planar three-body problem}
\author{Richard Moeckel}
\address{School of Mathematics\\ University of Minnesota\\ Minneapolis MN 55455}
\email{rick@math.umn.edu}

\author{Richard Montgomery}
\address{Dept. of Mathematics\\ University of California, Santa Cruz\\ Santa Cruz CA}
\email{rmont@ucsc.edu}

\author{H\'ector S\'anchez Morgado}
\address{Instituto de Matem\'aticas\\ Universidad Nacional Aut\'onoma
  de M\'exico\\ Ciudad de M\'exico C. P. 04510}
\email{hector@matem.unam.mx}

\date{May 1, 2017}



\begin{abstract}  
 Free time minimizers of the action  (called``semi-static'' solutions by Ma\~ne)
play a central role in the theory of weak KAM solutions to the Hamilton-Jacobi equation \cite{Fathi}.
We prove that any solution to Newton's  three-body problem which is 
asymptotic to Lagrange's parabolic homothetic solution is eventually a  free time minimizer.
Conversely, we prove that every free time minimizer tends to Lagrange's solution,
provided the mass ratios lie in a certain  large open set of
mass ratios. We were inspired by the work of   \cite{DM} who  had
shown that every   free time minimizer for the N-body problem is
parabolic, and therefore must be  asymptotic to the set of central configurations.  
We exclude being asymptotic to Euler's central configurations by a
second variation argument.  
Central configurations correspond to rest points for the McGehee blown-up dynamics. 
The   large open set of mass ratios are those for which   the linearized
dynamics at each    Euler rest point   has a complex eigenvalue.
\end{abstract}
\maketitle

\section{Introduction and Results.}

Lagrange's parabolic  homothetic solution (figure \ref{fig_Lag}) to the three-body problem consists of an
equilateral triangle  expanding at   the rate $t^{2/3}$. Analytically such a solution is  described by
  equation (\ref{cc}) below  with the  $c$  there  denoting the locations of  the vertices of an equilateral  triangle.

Newton's equations are   Euler-Lagrange equations for a well-known action.
Recently, Maderna and Venturelli \cite{MadernaV} proved that the
 parabolic homothetic Lagrange solution (figure \ref{fig_Lag}) is a 
  very strong type of action minimizer called a 
 ``  free time minimizer''  
 (or  ``semi-static'' by Ma\~ne \cite{Mane}) for this action.
 These minimizers are    objects of    central 
  interest in the recently developed   theory of weak KAM solutions
 \cite{Fathi}. 
 
   \begin{figure}[h]
\scalebox{0.4}{\includegraphics{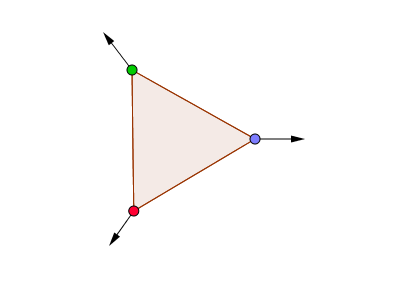}}
\caption{The Lagrange parabolic homothetic solution: an equilateral triangle expanding at the rate $t^{2/3}$.}
\label{fig_Lag}
\end{figure}

The 
 parabolic  homothetic central configuration solutions, of which Lagrange's solution is an example, are given by :
 \begin{equation}
  \label{cc}
  \ga_c(t)= \alpha c  t^{2/3},  
  c \text{ a central configuration, } \alpha = \alpha(c) \in \R  
  \end{equation} 
Central configurations  are  special configurations
which play an important role in the three-body problem.
Modulo  rigid motions and  scaling there are exactly five central configurations for 
 any given mass ratio. (See equation   \eqref{ccc} in section \ref{setup}  for   the defining equation of central configurations.)   Two are the  Lagrange   equilateral triangle configurations just described.  The Lagrange central configurations  count as two because there are two orientation types  of labelled equilateral triangles, related to each other by  reflection but not by rigid motion.  The
 remaining three central configurations are   the  Euler configurations  for which   the   three bodies lie on a line.   
 We index the Euler  configurations   by which   body lies   between the other two.

As time $t\to \infty$ in equation~(\ref{cc}) the velocity of each  body tends to zero.
Solutions with this property are called ``parabolic''.   
Although Chazy (\cite[chapter III]{Chazy}) did not use this definition, his results show that any  parabolic solution $\gamma(t)$  asymptotes
to  one of the  
$\ga_c$ of equation~(\ref{cc}) in the sense that
\begin{equation}
\label{asymptote}
\lim_{t \to \infty} t^{-2/3} \gamma(t)  = c.
\end{equation}
(See subsection \ref{parabolics} for a proof of an equivalent statement.)

As stated above, the   Lagrange solution is a FTM -- a  free time minimizer, and it  is parabolic.
 Da Luz and Maderna
\cite{DM} proved that every FTM is  parabolic.
These  results beg us to ask three questions.  Are the Euler parabolic homothetic  solutions FTMs?  
 To which of the five types  of central configurations may a FTM
  be asymptotic?  
Among the parabolic solutions, which ones are FTMs? 
We give a complete answer to these questions   for a large range of masses which we call the ``spiraling range'',
 depicted in figure~\ref{fig_masses}.  See subsection \ref{spiral_range} and    definition \ref{spiraling} for details on this range of masses.

 \begin{theorem}  
 \label{main}  In the spiraling range of mass ratios, every   free time minimizer 
$\gamma(t)$ for the planar  three-body problem tends to some Lagrange  configuration:  the limit $c$ of equation (\ref{asymptote}) is an   equilateral triangle.  Equivalently, these orbits lie in the stable manifold of one of the
Lagrange restpoints at infinity, as described in subsections \ref{spiral_range} and  \ref{parabolics}. 
\end{theorem}

\begin{figure}[h]
\scalebox{0.4}{\includegraphics{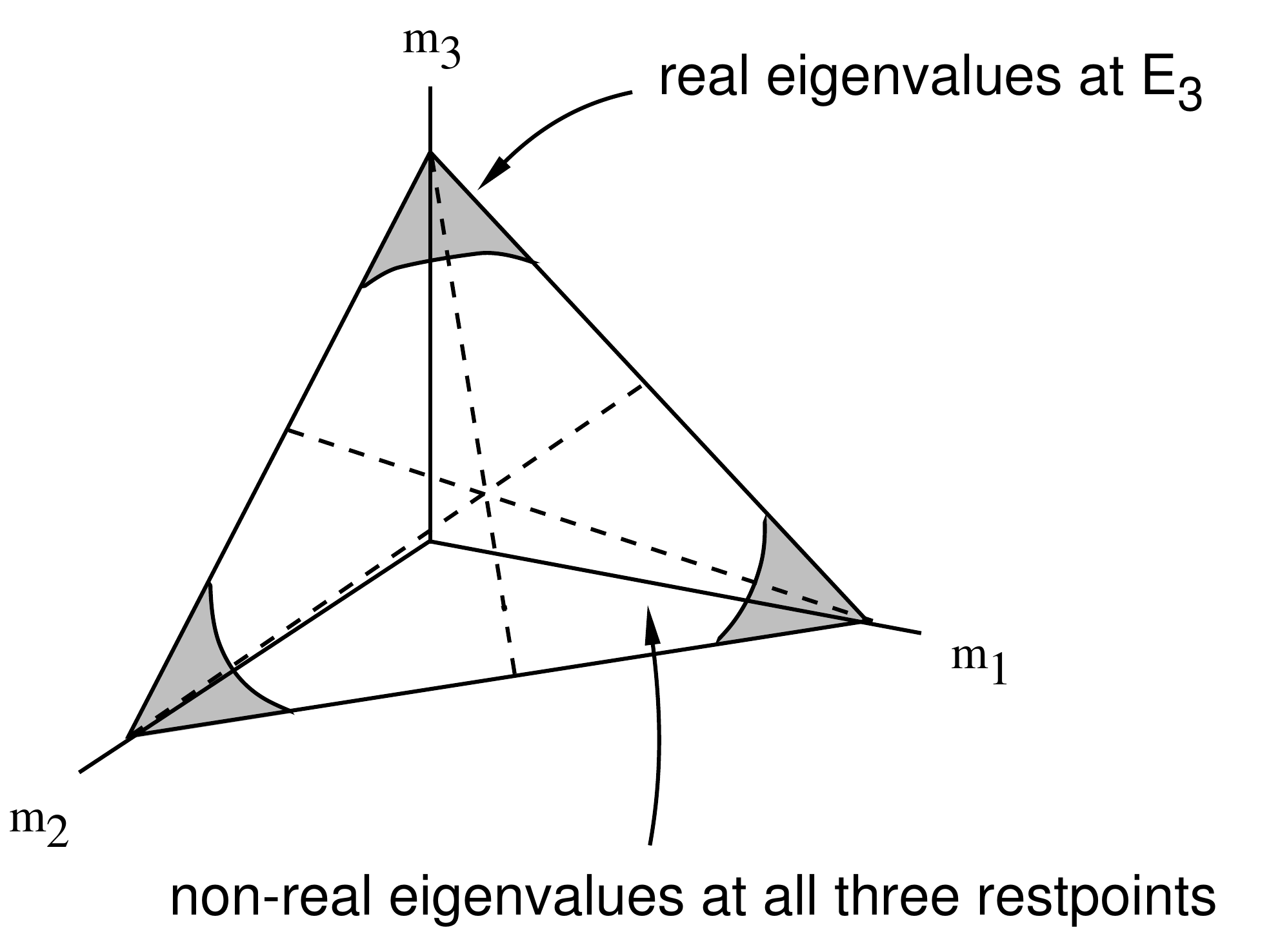}}
\caption{The spiraling range of mass ratios,  depicted  in the mass simplex $m_1+m_2+m_3=1$.   If $E_i$ denotes the collinear central configuration with mass $m_i$ in the middle, then the mass for which spiraling does not occur occupies the  small shaded region near the corresponding vertex, a region in which where  $m_i$ is much larger than the other masses.  The spiraling range, where all three central configurations have nonreal eigenvalues, is represented by the large unshaded region of the simplex. The figure is based on results in \cite{RobinsonSaari}.
}
\label{fig_masses}
\end{figure}

We  have the following converse to this theorem,  valid for all mass ratios:

\begin{theorem}
\label{converse}
Every parabolic solution  $\gamma(t), 0 \le t < \infty $  asymptotic to a  Lagrange  configuration
is a  free time minimizer upon its restriction to a sub-interval $T \le t <  \infty$, 
$T$ large enough. 
\end{theorem}

 \subsection{McGehee blow-up and  the spiraling range}
 \label{spiral_range}
 The   spiraling range  condition of theorem 1 
  first arose in Siegel's study of the triple collision singularity \cite{Siegel, SM}.   At a triple collision in backward time with $t\into 0^+$, the three mutual distances tend to  {\em zero} at the same $t^{2/3}$ rate.  (In fact the homothetic solutions (\ref{cc}) exhibit this behavior. ) After rescaling by $t^{2/3}$, these solutions also converge to central configurations.   
  Certain eigenvalues computed by Siegel  determine the rate of convergence and behavior of nearby noncollision orbits.
    For the case of an Eulerian central configuration, these eigenvalues are often complex which produces an oscillation.
   
In 1974, McGehee revolutionized  celestial mechanics by  
    a change of variables which partially compactifies
 phase space,  adding a   ``collision manifold'' \cite{McGehee, McGehee2}.   
 (We review McGehee's coordinates in subsection \ref{subsec:McGehee} below.)
Before compactification the flow has  no fixed points.  
After compactification the flow  admits  fixed points on the  collision manifold, indeed a pair of fixed points
for each central configuration, so 10 fixed points in all, modulo rotation.   Orbits converging to triple collision in forward or backward time 
now constitute the stable and unstable manifolds of these restpoints.  The eigenvalues of these restpoints are the same, up to scaling, as those found by Siegel.

A variation on McGehee's method allows us to study parabolic orbits in a nearly identical way.  We get a  flow at infinity
  identical to that of the collision manifold flow.  The  
  restpoints -- still parameterized by  the   central configurations,  are now interpreted in terms of parabolic solutions.
  The   eigenvalues of the linearized flow  at these restpoints are those computed by Siegel.   
  Their    stable   manifolds are comprised  of the parabolic orbits.   
For  the Euler central configurations  there is an open connected range of  
mass ratios for which we get 
eigenvalues with nonzero imaginary parts.  
We will  call this open set of mass ratios  the 
``spiraling range'' of mass ratios (see figure~\ref{fig_masses}.
See definition \ref{spiraling}.

{\bf Remark.} Following McGehee's work, several authors, including one of us, have used the oscillatory phenomenon  near triple collision 
to prove existence of interesting near-collision orbits \cite{Devaney, Simo, Moeckel1,Moeckel2,Moeckel3,Moeckel4}.   In these it was discovered that  if the mass ratios are  in the  spiraling range  
then certain  complicated chaotic behaviors, described by symbolic dynamics, are guaranteed to 
occur  for the corresponding three-body  problem.   

\section{The key lemmas}
\bigskip 
\subsection{For theorem 1} In trying to understand    free time minimizers, it is natural to consider 
the second  variation  of the action along a 
solution segment.
If this second variation is   negative  
then   we call the corresponding solution ``variationaly unstable''.  
The action of a variational unstable solution can  be decreased by  deforming  the solution curve 
in the direction of the negative variation, so such solution  curves cannot be  free-time minimizers.
Thus,  Theorem~\ref{main} is an immediate corollary of
\begin{lemma}\label{main-lemma}
  \begin{enumerate}[(A)]
  \item 
  The Euler parabolic homothety solution is   variationaly unstable if and only
if the mass ratio associated to that Euler central configuration is  spiraling (definition \ref{spiraling}) . 
\item 
For these same mass ratios,   any parabolic solution asymptotic to that Euler parabolic homothety solution
 is variationaly unstable.
  \end{enumerate}
\end{lemma}
\noindent
(Part (B) applies to the solutions in the stable manifold of the corresponding Eulerian restpoints at infinity.)

For the proof of lemma \ref{main-lemma}, the interested  reader may skip directly to   section \ref{sec:pfLemma1}.
 This reader may need   to refer back to section 4 for eigenvalue information. 

\subsection{For theorem 2}  The key lemma for Theorem 2 is  \begin{lemma}
\label{Lag_thm}
The stable manifold of a Lagrange restpoint $c$ at infinity is an immersed  Lagrangian
 submanifold which is a smooth embedded graph near $c$.   
\end{lemma}

{\bf Remark.} We use McGehee blow-up coordinates (see subsection \ref{subsec:McGehee}) to define  what we mean 
here by ``neighborhood of a restpoint $c$ at infinity'.
In those coordinates , a neighborhood  for the restpoint $c$ at infinity consists of 
the points $(u,c)$ of the form $0 < u < \delta$ ,  $|s -c| < \delta$. 
See  the definition \ref{def:nbhdInfinity} towards the end  of section \ref{sec:Lag}.

The proof of   theorem 2  combines this lemma with the theory of weak KAM
solutions, a theory to be reviewed in  the next section.    The basic  idea is as follows. 
Percino and Sanchez \cite{P-SM} constructed a weak KAM solution  for
each Lagrange parabolic homothetic solution.  Where smooth, the graph
of the differential of this solution forms  a Lagrangian 
manifold. The main point of the proof is that  this Lagrangian
submanifold  agrees with the stable manifold of lemma \ref{Lag_thm} near infinity.  
The full proof of the theorem is found in the last section of this
article, section \ref{sec:Lag}, which also contains the proof of lemma 2.

\section{ Free Time Minimizers, geodesics, and weak KAM.  }
\subsection{Free Time Minimizers and Jacobi-Maupertuis metric}
\label{sec:free-time-minimizers}
We define  free time minimizers.  We show they are minimizing  
geodesic rays  for the zero-energy Jacobi-Maupertuis metric

A {\it ray}  in a Riemannian manifold  is  a  geodesic 
whose domain is a half-line, and whose restriction to any
closed sub-interval of the domain is a minimizing geodesic
{\it between its endpoints}.  

 Ma\~n\'e generalized the notion  of rays and  lines  to Lagrangian dynamics
 defining ``semi-static'' curves. 
In the context of the N-body problem Maderna started calling
 semi-static curves {\it free time minimizers} and we stick to his
 terminology.

Consider the Newtonian N-body potential $U:\E\to]0,\infty]$,
$\E=(\R^d)^N$ given by
\[U(x)=\sum_{i<j}\frac{m_im_j}{r_{ij}}\]
for a configuration $x=(r_1,\ldots,r_N)\in\E$ of $N$
punctual positive masses $m_1,\ldots,m_N$ with $r_{ij}=|r_i-r_j|$,
as well as the Lagrangian $L:\E\times\E\to]0,\infty]$
\[L(x,v)=K(v)+U(x)=\frac 12\sum_{i=1}^Nm_i|v_i|^2+U(x).\]
The action of an absolutely continuous curve $\ga:[a,b]\to\E$ is 
given by
\[A_L(\ga)=\int_a^bL(\ga(t),\dot\ga(t))dt\]
\begin{definition}  
 \label{defMin1}A curve $\gamma_* :[a,b] \to\E$ with domain a closed bounded interval
 is called {\rm free time minimizer} if $A(\gamma_*) \le A(\gamma)$
  for every other curve $\gamma:[ c, d] \to\E$
 sharing its endpoints:  $\gamma_* (a) = \gamma(c),  \gamma_* (b) = \gamma(d)$.
 \end{definition} 
 
 {\bf Note: $d- c \ne b-a$ is allowed in the above definition.} 
 
 \begin{definition}  
  \label{defMin2}
 A curve $\gamma$ with domain an unbounded interval $J$
 (so $J$ is of the form  
 $(-\infty, \infty), [t_0, \infty)$ or $(-\infty, t_0]$)  is a {\rm  free time minimizer}
 if,  for each closed bounded sub-interval  $[a,b] \subset J$ , the restriction $\gamma|_{[a,b]}$ of
 $\gamma$ to $[a,b]$ is a free time minimizer in the sense above (def. \ref{defMin1}). 
 \end{definition}
 
  \begin{prop}[Basic facts regarding free time minimizers]\label{factsFTMs}\quad
\begin{enumerate}[(i)]
 \item If $\gamma$ is a  free time minimizer then its total energy $H=K-U$ is zero at each time. 
 \item  The  free time minimizers are precisely the   minimizing
 geodesics for the zero energy Jacobi-Maupertuis [JM]  metric 
 \begin{equation}
 \label{JM}
 \langle v, v \rangle_q= 4 U(q) K(v)
 \end{equation} on $\E$, with these geodesics reparameterized so as to
 have $0$ energy.   
 \end{enumerate} 
 \end{prop} 
 
 \begin{proof}
 See \cite{DM} for a proof of fact (i).
  We now prove fact (ii).   The   inequality $a^2 + b^2 \ge 2ab$
 with equality if and only if $a = b$ yields
 $$ 2 \sqrt{U} \sqrt{K} \le L = K+ U \text{ with equality iff  }  H = 0.$$
 Consequently, for any  absolutely continuous curve $\gamma: [a, b] \to\E$ we have
 $$\ell_{JM} (\gamma) \le A(\gamma) \text{ with equality iff  } H = 0 \text{ along } \gamma$$
 where $\ell_{JM} (\gamma)$ is the JM length functional.
 Now, if $\gamma$ is a free time minimizer, we have that $H=0$ and so
 the Jacobi length  and action are the same.  Thus $\gamma$ must be a
 minimizing geodesic for the JM metric. 
 \end{proof}
 
 Note that it follows from the proposition that 
 the FTMs with domain $[T_0, \infty)$  are precisely the rays for the zero energy JM metric.

 \subsection{Weak KAM and Hamilton Jacobi}
 \label{subsec:weakKAM}
 
It follows from (ii) of proposition \ref{factsFTMs}   that 
 $$d_{JM} (q, q_0) = \inf_{\gamma: q_0 \to q}  \int_{\gamma} L $$
 is the JM distance between the points $q, q_0 \in Q$. 
  Freezing $q_0$, we get  the function  $f(q) = d_{JM} (q,q_0)$ on $Q$.  
  It  is well known that 
 on a Riemannian manifold the gradient of the  distance function from a point (or a subvariety)  is  a unit vector wherever differentiable.
  Viewed in dual terms, this unit length gradient condition   reads 
 $ \| df (q) \| _{JM} = 1$ where  $\| df (q) \| _{JM} $ is computed relative to the
 dual metric on covectors which is induced by the JM metric. 
 This dual  metric   satisfies   $(\| df \|_{JM})^2  = \frac{ \|d f \|_K ^2 }{ 2 U }$ where the subscript $K$ is the length of a covector
 relative to the dual to the metric $K$.  
 In other words,
 \begin{equation}
   \label{eq:hjm}
\fr12 \| d f \|_K ^2= U,
 \end{equation}
 which is the Hamilton-Jacobi equation. Since
 \[df(q)(v)\le \fr12 \| d f \|_K ^2 + K(v),\]
the relation \eqref{eq:hjm} is equivalent to  
$$df(q)(v) \le L(q,v) \text{ for all } q, v \in\E,$$
 with equality realized for some $v$.
 
 Now set $v = \dot \gamma$ for some curve $\gamma(t), t \ge 0$ and integrate:
 \begin{equation}
 \label{wKAM1}
 f(\gamma(a)) - f (\gamma(0)) \le \int_{\gamma[0,a]} L dt
 \end{equation}
A curve $\gamma:[0,a]\to\E$ is called {\it calibrated by} $f$ if
  \begin{equation}
 \label{wKAM2} 
 f(\gamma(a))-f(x)=\int_{0}^{a} L (\gamma,\dot \gamma) dt  .
 \end{equation}

 \begin{definition} \label{wkam}
A function $f: \E \to \R$ which  satisfies equation \eqref{wKAM1} for
any curve $\ga:[0,a]\to\E$ is called {\rm dominated}. 
If moreover for all $x  \in \E$  there exists a curve   $\gamma:[ 0, \infty) \to \E$
such that $\gamma(0) = x$ and for all $t > 0$, $\ga|[0,t]$
is calibrated by $f$, the function is called  {\rm (forward) weak KAM solution}.

\end{definition} 

 It follows from the definitions that any curve $\gamma$ calibrated
 by a dominated function is a  free time minimizer.  

{\bf Remark 1.}  Any weak KAM solution satisfies the classical Hamilton-Jacobi equation 
at differentiability points. 
(See  Fathi \cite{Fathi} for details and a careful  exposition.)

 {\bf Remark 2.} In a neighborhood of a point where $f$ is smooth, the graph of it  differential
 $df$ forms a smooth Lagrangian graph sitting in the zero energy level of the cotangent bundle.
 This graph is   invariant under the Hamiltonian
 flow and the  solutions to Hamilton's equation foliate the graph.  If we write such a solution
 as $(c(t), p(t))$ then necessarily $c$ is a piece of a calibrating curve, and
 $p(t) = \dot c (t)^*$ is the Legendre transformation of the derivative of that calibrating curve.

{\bf Remark 3.}   There are also backwards weak KAM solution. 
For these, the calibrated curves of equation~(\ref{wKAM2}) are parameterized as 
 $\gamma: (-\infty, 0] \to \E$, and end at $x$:    $\gamma(0) = x$.

\vskip .3cm 
Buseman found a nice way  to construct a weak KAM solution out of 
a  free time minimizer. 
 \begin{definition}
 \label{defBuseman} The Buseman function $B_\ga$
associated to a fixed   free time minimizer $\ga:[0,+\infty)\to \E$ is 
 \[B_\ga(x) 
       =\lim_{t\to+\infty}\left[d_{JM}(\ga(0),\ga(t))-d_{JM}(x,\ga(t))\right]\]
      \end{definition}
 Buseman's  inspiration came from horocycles in hyperbolic geometry.
These horocycles are the level sets of the Buseman function associated to a hyperbolic geodesic $\gamma$.
Percino and Sanchez \cite{P-SM}  re-expressed Buseman's  idea in the present language, as we have just done,  
and were able to prove
\begin{prop}
\label{Buseman} \cite{P-SM} The Busemann function $B_c$ 
associated to any parabolic Lagrange solution $\ga_c$ is
a weak KAM solution.  Moreover,  for any configuration $x$,
the corresponding calibrated curve of equation~(\ref{wKAM2}) will be asymptotic to $c$.
\end{prop} 
This proposition will be  central to proving  Theorem~\ref{converse}

\section{Set-up}
\label{setup}
 
 The three-body  configuration space is three copies of the   Euclidean plane.
 We write points, or ``configurations'', as $q = (q_1, q_2, q_3),  q_i \in \R^2 \cong \C$.
 and we write velocities as $v = (v_1, v_2, v_3), v_i \in \R^2 \cong \C$.
 The masses are $m_i > 0$. 
 By a standard trick from introductory physics  we can   
 assume that the center of mass of the configuration  is zero:
 \begin{equation}
 m_1 q_1 + m_2 q_2  + m_3 q_3 =0
 \label{cOfm}
 \end{equation}
 and  that the  total linear momentum is also zero:
 \begin{equation} m_1 v_1 + m_2 v_2  + m_3 v_3 =0.
 \label{linmomentum}
 \end{equation}
 We write $\E \cong \R^4  \subset \R^2 \times \R^2 \times \R^2$
 for either  four-dimensional linear subspace.

  We introduce the mass metric: 
  \[ \langle v , w \rangle_m : = m_1 v_1 \cdot w_1 + m_2 v_2 \cdot w_2 + m_3 v_3 \cdot w_3 =v\cdot Mw\]
where $M=\diag(m_1,m_1,m_2,m_2,m_3,m_3)$ is the mass matrix,
so that the usual kinetic energy is
$$K = \frac{1}{2} \langle v , v \rangle_m ; v \in \E.$$
 We will also make use of the moment of inertia of the configuration with respect to the origin
  \begin{equation} 
  I(q) = \langle q , q \rangle_m.
 \label{momentofinertia}
 \end{equation}
 An alternative, translation invariant formula for the moment of inertia is
$$I(q) = (m_1m_2 r_{12}^2+m_1m_3 r_{13}^2+m_2m_3 r_{23}^2)/m\qquad r_{ij}=|q_i-q_j| \qquad m = m_1 + m_2 + m_3.$$

The Lagrangian  for the three-body problem  
$$L: \E \times \E \to]0,\infty]$$
is given by  
\[ L (q,v) = K(v) +   U(q)\]
  where
$$U(q) = \frac{m_1 m_2}{|q_1 - q_2 |} +  \frac{m_2 m_3}{|q_2 - q_3 |} +  \frac{m_1 m_3}{|q_1 - q_3 |}$$
is the negative of the potential energy.  Newton's equations read
\begin{equation}
 \ddot q = \nablam U (q) 
 \label{Newton}
 \end{equation}
where the gradient is with respect to the mass-metric: $dU(q) (v) = \langle \nablam U (q), h \rangle_m$.  
The distance from the origin (which is triple collision) with respect to the mass metric is denoted by $r$:
\[ r(q) = \sqrt{I(q)}= \sqrt{ \langle q , q \rangle_m }.\]
Each configuration $q\in\E$ determines a unique {\em normalized configuration} $s = q/r$ with $r(s)=1$.

A {\em central configuration or CC} is a point  $q\in\E$ such that
\begin{equation}
\label{ccc}
\nablam U (q) +  \l q = 0.
\end{equation} 
Using the homogeneity of the potential it is easy to see that
$\l = \frac{U(q)}{r(q)^2}$ which reduces to $\l = U(c)$ for a normalized central configuration.
The homothetic parabolic motion associated to  such a normalized central  configuration $c \in \E$ is 
\begin{equation}
\label{ccr}
\gamma_c (t) = \Big(\frac 92 U(c)\Big)^\frac 13c  t^{2/3}
\end{equation}
and is an exact solution to Newton's equations.  

\subsection{McGehee coordinates}
\label{subsec:McGehee}
Following McGehee, define new variables
$$r= \sqrt{\langle q,q\rangle_m}\qquad s = \fr{q}{r} \qquad z= r^{\fr12}\dot q.$$
The normalized configuration variable $s=(s_1,s_2,s_2)$ satisfies
$I(s) = r(s)^2  = 1$.

Introduce the function
 \begin{equation} \label{McGehee3}
 v = \massmetric{z}{s} 
 \end{equation}
 and define  a  new time variable $\tau$ by 
$\dfrac{d}{d \tau} =r^\frac 32  \dfrac{d}{dt}$ 
and write $f^\prime = \dfrac{df}{d\tau}.$  
Then the new, blown-up variables satisfy
\begin{align}
\label{eq:mcgeheer}r^\prime  &= vr\\
\label{eq:mcgehees}s^\prime &= z-vs\\
\label{eq:mcgeheez}z^\prime &= \nablam U(s) + \fr12 vz.
\end{align}
\noindent
In deriving these one uses the homogeneity of the potential to see that $U(q) = r^{-1}U(s)$ and $\nablam U(q) =  r^{-2}\nablam U(s)$.
The equations now make sense when $r = 0$: the  triple collision singularity has been blown-up into the  invariant manifold   $\{r=0\}$ and the differential equations for $(s,z)$ are independent of $r$.

In these coordinates, the eight-dimensional  phase space $X =\E\times \E $ is given by the system of equations 
\begin{equation}\label{eq:normalize}
m_1 s_1 + m_2 s_2  + m_3 s_3 =0, \qquad m_1 z_1+m_2 z_2+m_3 z_3=0\qquad \text{ and }  \langle s,s\rangle_m=1.
\end{equation}
The flow preserves $X$ as well as  the energy levels $H(q,p) = h$ which are now given by
$$H(s,z) = \fr12 \massmetric{z}{z} -U(s) = rh.$$
We will be especially interested in the case $h=0$.

The rate of change of $v=\massmetric{z}{s} $ satisfies
$$v^\prime = \massmetric{z}{z} -U(s) -\fr12v^2 = \fr12 \massmetric{z}{z} - \fr12v^2+rh.$$
If $r=0$ (triple collision) or $h=0$ (zero energy), this simplies to
$$v^\prime =\fr12(\massmetric{z}{z} -v^2) \ge 0$$
where the nonnegativity follows from the Cauchy-Schwarz inequality. 
Thus $v$ is a Lyapunov function on the triple collision ($r=0$) and zero energy ($h=0$) submanifolds.

When $h=0$ we  study the  motions with $r\into\infty$ by  replacing   $r$ by $u = r^{-1}$.  Then
equation \eqref{eq:mcgeheer} is replaced by: 
\begin{equation}\label{eq:mcgeheeu}
u^\prime  = -vu
\end{equation} 
while   equations \eqref{eq:mcgehees}, \eqref{eq:mcgeheez} and the energy 
equation $H(s,z)=0$ remain unchanged.  Now $\{u=0\}$ is invariant and represents the 
dynamics at infinity for the zero energy problem.

\subsection{Restpoints}\label{equilibria}
A point $(s,z)$ is an equilibrium point for the differential equations 
\eqref{eq:mcgehees}, \eqref{eq:mcgeheez} if and only if
$$v^2 = 2U(s) \qquad z= v s$$
and
\begin{equation}\label{eq:CC}
\nablam U(s) +U(s)s = 0
\end{equation} 
which is exactly the equation (\ref{ccc}) for a normalized central configuration.

Equation (\ref{eq:CC}) can also be viewed in another way. The normalization condition $\langle s,s\rangle_m=1$ defines a three-sphere
 $\mathcal{E}\subset \E$.  Then (\ref{eq:CC}) is the equation for critical points
 of the restriction of $U(s)$ to this sphere.  In fact, the equations  can be written $\tilde\nabla U(s) = 0$ where
\begin{equation}\label{eq:gradtilde}
\tilde\nabla U(s) =  \nablam U(s) +U(s)s.
\end{equation} 
is  the gradient  
of the restriction of $U$ to the three-sphere  with respect to the  metric on this sphere induced by   the mass metric.  
(The restricted gradient of a general smooth    $f: \E \to \R$  is  $\tilde \nabla f (s) = (\nabla_m  f (s)) ^{T}$,
where $v^{T}$ is the orthogonal projection of the vector $v \in \E$
to the tangent space at $s$ to the sphere.  Using Euler's identity
for homogeneous functions, we find that if   $f$ is homogeneous of degree $\alpha$ then $(\tilde \nabla  f (s))  = \nabla_m  f - \alpha f(s) s$, hence the expression for $\tilde\nabla U$.)  
 Because of the rotational symmetry of the potential, there are actually five circles of critical points, one for each of the five central configuration shapes.

Each normalized central configuration $s_0$ determines two equilibrium points in the triple collision manifold $(0,s_0, z_0)$
where
\begin{equation}\label{eq:equilibrium}
z_0 = v_0 s_0\qquad v_0 = \pm \sqrt{2U(s_0)}.
\end{equation}
For the zero energy problem we also get two equilibrium points at infinity with the same $(s_0, z_0)$ and $u=0$.  

For a given normalized  central configuration $c$, its equilibria at collision and  infinity are connected by the zero energy parabolic homothetic orbits of equation~(\ref{cc}), or, what is the same, equation ~(\ref{ccr}).  These are precisely the zero energy  solutions such that $r$ changes while the
$(s,z)$ remains at their equilibrium values (\ref{eq:equilibrium}). The size $r$ is given by
$$r(\tau) = \exp( v_0 \tau) \qquad\text{or}\qquad u(\tau) = \exp( -v_0 \tau).$$
Here $\tau$ denotes the normalized time variable and $v_0 = \langle s_0, z_0 \rangle$. 

\subsection{Stable and unstable manifolds}\label{sec_stableunstable}
As just discussed the parabolic solutions are precisely the solutions in the stable manifold of 
one of the rest points at infinity.    For the three-body problem, these rest points are all hyperbolic (after allowing for rotational symmetry) and their Lyapunov exponents will play an important role in what follows.  These exponents  have been calculated before \cite{Siegel,SM}  but we will present the results here (with some details relagated to an Appendix) for the sake of completeness and to correct some unfortunate typos which appeared in \cite{Moeckel4}.

Consider the variational equation of the blown-up differential equations \eqref{eq:mcgeheer}, \eqref{eq:mcgehees}, \eqref{eq:mcgeheez}  at one of the equilibrium points
$p = (0,s,z)$.  Differentiation and evaluation at $r=0$ gives the $13\times 13$ matrix:
$$
A =\m{v&0&0\\&&\\
0&-vI -s z^tM &I - s s^t M \\&&\\
0&D\nablam U(s)+\fr12 z z^t M &\fr12 v I +  \fr12 z s^t M}.
$$
For a restpoint at infinity, that is to say $u=0$,  the only difference in $A$ is that the upper left $v$ of the matrix becomes $-v$ so the two cases can be considered together.  
 
Some words may be helpful   regarding the  terms 
  $sz^t M$, $s s^t M$, $z z^t M$ and $z s^t M$. 
The term $s z^t M$ in the 2-2 block of $A$, for example, 
 describes  the linear operator taking $\delta s$ to $s \langle z , \delta s \rangle_m$. 
All  these terms arise from  linearizing the quadratic function 
 $v = \langle s, z \rangle_m = s^t M z$ which   occurs as a factor in equations \eqref{eq:mcgehees}, \eqref{eq:mcgeheez} , and \eqref{eq:mcgeheer}.
 
Let $(\d r, \d s, \d z)\in T_pX \subset \R^{13}$ denote a tangent vector to  $X$ at $p$, where $X$ is our eight-dimensional  phase space
defined by the  normalization equations (\ref{eq:normalize}).  Linearizing the first and last of the normalization equations  we find
\begin{equation}\label{eq:dnormalize}
m_1 \d s_1 + m_2 \d s_2  + m_3 \d s_3 = 0 \text{ and  }  s^t M\d s = 0
\end{equation}
which    defines the three-dimensional tangent space to the sphere  $\mathcal{E} \subset \E$ at $s$.
Since $z = v s$ we also have 
$$z^t M \d s = 0$$
so we can ignore the terms involving $z^tM $ in the second column of $A$.  In general, it is not true that
$s^tM \d z=0$, however,   this equality does hold  for vectors  lying in $T_p X$ and also  tangent to the energy manifold $H(s,z)=0$, since  for such vectors 
$$\d H = z^t M \d z - \nabla U(s)\cdot \d s = v s^t M \d z +U(s)\, s^tM \d s =  v s^tM \d z .$$
Thus vectors with $\d H=0$ also have $s^tM \d z = 0$ and for these the third column of $A$ also simplifies.

One easily checks that the  vectors $(\d r, \d s, \d z) = (1,0, 0)$ and $(\d r, \d s, \d z) = (0,0,s)$ are eigenvectors in $T_pX$ with
eigenvalues $\l_1= v$ and $\l_2 = v$.  The first vector satisfies $\delta H = 0$ while the second vector satisfies $\delta H = v \ne 0$.   
The subspace $\d r = \d H =0$ is a 6-dimensional subspace of $T_p X$
 invariant under $A$. It follows that the other 6 eigenvectors of $A$ restricted to $T_pX$  must lie in this subspace.  Dropping the $z^tM $ and $s^tM$ terms from $A$ we find that the other eigenvectors are of the form 
$(0, \d s, \d z)$ where $(\d s, \d z)$ is an eigenvector of the $12\times 12$ matrix
$$
B= 
\left[\begin{array}{cc}
-vI&I\\
D\nablam U(s)&\fr12 v I
\end{array}\right]
$$

The following lemma  (see \cite{Devaney_structural}  for the lemma's origin) gives the eigenvectors and eigenvalues of $B$ in terms of those of $D\nablam U(s)$ or equivalently those of \footnote{From equation ~\ref{eq:gradtilde}
we get that   $D\tilde\nabla U$ equals the expression of  equation~\ref{eq:dgradtilde} plus  the term $\nabla U (s) \otimes s^t M$ which we  ignore since $\langle s, \delta s \rangle_m =0$.} 
\begin{equation}\label{eq:dgradtilde}
D\tilde\nabla U(s) = D\nablam U(s) + U(s)I.
\end{equation}
\begin{lemma}\label{lemma_lambda}
Let $s$ be a normalized central configuration and $z = vs$ where $v^2
= 2U(s)$.  If a vector $\d s$ satisfying (\ref{eq:dnormalize}) is an
eigenvector of $D\tilde\nabla U(s)$ with eigenvalue $\a$  then  the
vectors $(\d s, k_\pm \d s)$ are eigenvectors of $B$ with eigenvalues
$$\l_\pm = \fr{ -v \pm \sqrt{v^2+16 \a}}{4} \qquad k_\pm = v + \l_\pm$$
\end{lemma}
\begin{proof}
The assumptions imply that 
$$D\nabla_m U(s) \d s = (\a -U(s))\d s = (\a-\fr12 v^2)\d s.$$
 Then the product of $B$ and $(\d s, k \d s)$ is $( (-v+k)\d s, (\a-\fr12 v^2+\fr12 vk)\d s)$.   Setting this equal to $\l( \d s, k \d s)$ leads to the equations
 $$k = v + \l\qquad \l^2+ \fr12 v \l -\a = 0$$
 and solving for $\l , k$ completes the proof.
\end{proof}

Using the rotational symmetry, it is easy to 
guess one eigenvector of $D\tilde\nabla U(s)$ satisfying (\ref{eq:dnormalize}).  Let $s^\perp = (s_1^\perp, s_2^\perp, s_3^\perp)$ denote the vector with each $s_i$ rotated by $90^\circ$ in the plane.  Since $U$ is rotationally invariant we have $D\tilde\nabla U(s) s^\perp = 0$
so $\d s= s^\perp$ is an eigenvector with $\a = 0$.  There are two more eigenvectors satisfying (\ref{eq:dnormalize}) and they will determine what we will call the nontrivial eigenvalues of $D\tilde\nabla U(s)$.    These will be calculated in the appendix.  For now, we just record the results.

\begin{prop}\label{prop_evals}
Let $s$ be a normalized central configuration and $p = (r,s,z) = (0, s, v s)$ one of the triple collision restpoints where either $v = \sqrt{2U(s)}$ or $v=-\sqrt{2U(s)}$.  
Let $\a_1, \a_2$ be the two nontrivial eigenvalues of $D\tilde\nabla U(s)$.  Then the  eight Lyapunov exponents of the variational equations on $T_pX$ are
$$\l = v, v, -v,0, \fr{ -v \pm \sqrt{v^2+16 \a_1}}{4} ,\fr{ -v \pm \sqrt{v^2+16 \a_2}}{4}.$$
The eigenvalues for an equilibrium at infinity $p = (u,s,z) = (0, s, v Ms)$ are the same except the first one becomes $-v$.
\end{prop}

The Lagrangian (equilateral) critical points form circles of local minima in $\mathcal E$.  The corresponding nontrivial eigenvalues are both positive.
\begin{prop}\label{prop_Lagrangeevals}
The nontrivial eigenvalues of $D\tilde\nabla U(s)$ at an equilateral central configuration are
$$\a_1,\a_2 = \fr{3U(s)}{2}\left(1\pm\sqrt{k}\right)$$
where
$$k=\fr{(m_1-m_2)^2+(m_1-m_3)^2+(m_2-m_3)^2}{2(m_1+m_2+m_2)^2}.$$
The four corresponding nontrivial eigenvalues at one of the Lagrangian equilibrium points at triple collision or at infinity are
$$\l = \fr{-v}{4}\left( 1\pm \sqrt{13\pm 12\sqrt{k}}\right)$$
\end{prop}

After allowing for the rotation, the  Eulerian, collinear critical points are saddles with one positive and one negative nontrivial eigenvalue.  Their values depend on the shape of the configuration.  Consider the collinear central configuration with $m_2$ between $m_1$ and $m_3$.  Instead of normalizing the configuration we can look for critical points of the translation and scale invariant function
$$F(s) = I(s)U(s)^2$$
with no constraints.  It is easy to see that if $s$ is a critical point of $F$ then the corresponding normalized configuration satisfies (\ref{eq:CC}).
Using $F$, we may assume without loss of generality that
$$s_1=(0,0)\qquad s_2= (r,0)\qquad s_3=(1+r,0)$$
where $0<r<1$.  This gives a function of one variable $F(r)$ and setting $F'(r)=0$ leads to the fifth degree equation
\begin{equation}
\label{eq:g}
\begin{aligned}
g = (m_2+m_3&)r^5+(2m_2+3m_3)r^4+(m_2+3m_3)r^3\\
&-(3m_1+m_2)r^2-(3m_1+2m_2)r-(m_1+m_2)=0.
\end{aligned}
\end{equation}
There is a unique root with $0<r$ by Descartes' rule of signs which will determine the shape of the collinear CC.

\begin{prop}\label{prop_Eulereevals}
The nontrivial eigenvalues of $D\tilde\nabla U(s)$ at the collinear central configuration with $m_2$ between $m_1,m_3$ are
$$\a_1,\a_2 = -U(s)\nu,\, U(s)(3+2\nu)$$
where
\begin{equation}\label{eq_nu}
\nu = \fr{m_1(1+3r+3r^2)+m_3(3r^3+3r^4+r^5) }{ (m_1+m_3)r^2  + m_2(1+r)^2(1+r^2)  }
\end{equation}
and $r$ is the positive root of (\ref{eq:g}).
The four corresponding nontrivial eigenvalues at the Eulerian equilibrium points at triple collision or at infinity are
\begin{equation}
  \label{eq:eigen-euler}
\l = \fr{-v}{4}\left( 1\pm \sqrt{1-8\nu}\right), \fr{-v}{4}\left( 1\pm \sqrt{25+16\nu}\right).  
\end{equation}
 The values at the other Eulerian restpoints are found by permuting the subscripts 
on the masses.
\end{prop}

Note that  the Eulerian restpoints have a pair of nonreal eigenvalues if and only if $\nu > \fr18$.  

\begin{definition} \label{spiraling} (i) We say that the Euler configuration with mass 2 in the middle is {\it spiraling} 
if  $\nu > \fr18$  with $\nu$  as in proposition \ref{prop_Eulereevals}.

(ii) If each of the three  Euler configurations is spiraling then we say that the mass ratios $[m_1:m_2:m_3]$ are in the ``spiraling range''.
\end{definition}

 Figure~\ref{fig_masses} shows the masses for which $\nu>\fr18$ for each of the three Eulerian restpoints.

\subsection{Parabolic motions tend to rest points at infinity} \label{parabolics}
The  qualitative study of parabolic  solutions
goes   at least back to Chazy \cite{Chazy} (particularly  chapter 3).
We now show that  the stable manifolds of the rest points at infinity are precisely the unions 
classical parabolic solutions, a fact well known to experts.

We will use the following weak form of the definition of ``parabolic''.

\begin{definition} A solution to the three-body problem is (future) parabolic if  the solution's domain contains a positive  half line $[t_0, \infty)$ and if the Newtonian velocities
of all three  bodies  tend to zero
as  (Newtonian) time tends to infinity. 
\end{definition}

{\bf Remark.}  Define ``past parabolic'' by letting time tend to negative infinity.  
We stick with future parabolic for simplicity.  


\begin{prop}
\label{prop:parabolic}
Any parabolic solution has energy $0$ and lies in the stable manifold
of one of the rest points at infinity.    Conversely every solution in the stable
manifold of a rest point at infinity and such that $u>0$ is a parabolic solution. 

Alternatively: Let $R_+$ denote the collection of rest points at infinity for which $v>0$.
  Then $R_+$ is normally hyperbolic
and its stable manifold $W^s (R_+)$ is foliated by the stable manifolds $W^s (c)$ tending to the
rest point associated to the   central configuration $c$.  
Each parabolic solution  lies in some $W^s (c)$. 
\end{prop}
  \begin{proof} For the three-body problem, it seems that most of this result follows from Chazy's work \cite{Chazy}.  But he does not use our definition of parabolic.  For completeness we will give a proof here using ideas from \cite{MarchalSaari, Chenciner2}. 
 
 For a parabolic motion,  the kinetic energy
 $K(v)\into 0$ as $t\into\infty$.  The energy equation 
 $K(v) -U(q)=h$ and the fact that $U(q)>0$ imply that $h\le 0$.  To rule out the case $h<0$ we use the Lagrange-Jacobi identity
 $\ddot I(t) =  2K+ 2h$.  If $K\into 0$ and $h<0$ then $\ddot I(t)$ has a negative upper bound for $t$ sufficiently large which forces $I(t)\into 0$ (total collapse) in finite time.  Such a solution would not  exist for large $t>0$.
 
 It is also easy to see that $r(t)\into\infty$ as $t\into\infty$.  Indeed the energy equation gives $r(t)K(v) = r(t)U(q) = U(s)$.  Now the normalized potential $U(s)$ has a positive lower bound depending only on the choice of masses, so $K(v)\into 0$ implies $r(t)\into\infty$.
 
The main theorem in Marchal and Saari \cite{MarchalSaari} describes the asymptotic behavior as $t\into\infty$ for any solution of the $n$-body problem which exists for all $t\ge 0$.  For any such solution, either $r(t)/t\into \infty$ or else all of the position vectors satisfy
 $q_k = A_k t+ O(t^\fr23)$ for some constant vectors $A_k$, possibly zero.  The second case implies that either
 $r(t)/t \into L$ for some $L>0$ or else $r(t) = O(t^\fr23)$ (the latter holding when all $A_k=0$).  We will show that in fact we have $r(t) = O(t^\fr23)$ for parabolic orbits.
 To see this note that given any $\e>0$ there is $t_0$ such that $\ddot I(t) = 2K < \e$ for $t\ge t_0$.  Then
 $I(t) \le I(t_0)+ \dot I(t_0) (t-t_0) + \fr12\e(t-t_0)^2$.  
If $r(t)/t\into L\in (0,\infty]$ then $I(t)/t^2 \into L^2\in (0,\infty]$
 and we have the contradiction that $0<L^2< \fr12 \e$ for all $\e>0$.

Next we show that the quantity $v(t)$ appearing in the blown-up
equations tends to a finite limit $v(t)\into \bar{v}>0$ as $t\into\infty$.  
Recall that $v(t)$ is non-decreasing since $v$ is a Liapanov function
on the zero energy surface.  Also 
$$v(t) = r^{-1}r'(t) = \sqrt{r(t)}\, \dot r(t) =  \fr23 \fr{d}{dt} r(t)^{\fr32}.$$
Since $r(t)\into\infty$ we have $v(t)>0$ for $t$ large so either $v(t)$ approaches some $\bar{v}>0$ or else $v(t)\into\infty$.  But integration gives
$$\fr1{t}\left( \fr23r(t)^{\fr32} -  \fr23r(0)^{\fr32}\right) = \fr{1}{t} \int_0^t v(s)\,ds.$$
If $v(t)\into \infty$ as $t\into\infty$ we would get $r(t)^{\fr32}/t\into\infty$ contradicting $r(t) = O(t^\fr23)$.
 
 Finally we can use the dynamics on the infinity manifold to finish the proof.  First note that the estimate $r(t) = O(t^\fr23)$ shows that the rescaled time $\tau$ with $\dot\tau(t) = r^\fr32(t)$ satisfies $\tau(t)\into\infty$ as $t\into\infty$, so a forward parabolic orbit exists for all 
 $\tau\ge 0$ and we have $u(\tau) = 1/r(\tau) \into 0$ as $\tau\into\infty$.  We claim that the $\omega$ limit set of our parabolic orbit consists of one of the restpoints in the manifold $\{u=0\}$.
 
Consider the subset $S = \{(u,s,z): u=0, H(s,z)=0, v=\bar{v}\}$ where $\bar{v}$ is the limit of $v(\tau)$ for a certain parabolic solution. We will show that this solution avoids a neighborhood of the double collision singularities in $S$.  Since the solution exists for all $t\ge 0$ and $\tau\ge 0$ it does not actually have a double collision, but we want to avoid a whole neighborhood.  Let $w= s' = z-vs$, the component of $z$ tangent to the ellipsoid $\mathcal{E}$.  The energy equation can be written  $\fr12 v^2 +\fr12 |w|^2 = U(s)$.  Since $U(s)\into \infty$ at collision while $v^2$ is bounded near $S$, it follows that $|w|$ is large near collision.
Now $v' =\fr12(\massmetric{z}{z} -v^2)  = \fr12|w|^2$ while the arclength $\sigma$ in $\mathcal{E}$ satisfies $\sigma' = |s'| = |w|$.  Hence the rate of change of $v$ with respect to arclength is $\fr12 |w|$.  From this we see that there is some neighborhood $\mathcal{U}$ of the double collision singularities in $S$ such that any initial condition in $\mathcal{U}$ crosses into the set $v>\bar{v}$.  So our solution avoids $\mathcal{U}$.

We conclude that our parabolic solution must converge to the compact set $S' =S\setminus \mathcal{U}$ as $\tau\into\infty$.  Therefore it has a nonempty, compact $\omega$ limit set contained in $S'$.  For orbits in the limit set we must have $v'(\tau) = \fr12 |w(\tau)|^2 = 0$ for all $\tau$ and this happens only at the restpoints, that is, at the points $(u,s,z) = (0,c,vc)$ where $c$ is a central configuration.    For the three-body problem, the restpoints form five circles and the eigenvalue computations show that these are normally hyperbolic invariant manifolds.  It follows that the omega limit set consists of just one of the restpoints.
 
The converse is easier.  Assume that $r(\tau)\into\infty$ and $(s(\tau),z(\tau))\into (c,\bar{v}c)$ where $\bar{v}=\sqrt{2U(c)}$.   Since
$K(z) = \sqrt{r}K(v)\into U(s)$ we have $K(v)\into 0$ as $\tau\into\infty$.  Inverting the change of timescale we find that $t\into\infty$ as $\tau\into\infty$ so the solution is parabolic.
  
 \end{proof}

\section{ Proof of Lemma \ref{main-lemma}}
\label{sec:pfLemma1}
In this section we revert to the classical variables and timescale but will make use of the eigenvalue computation of proposition~\ref{prop_Eulereevals}.

\subsection{ Proof of part (A) of  Lemma \ref{main-lemma} }
\label{sec:1A}

.
\begin{proof}

Consider perturbations $\ga^\ep$ of a  homothetic parabolic motion $\gamma_c$ associated to  a central configuration $c$ 
as in \eqref{ccr}, so $r(c) =1$, $\ga_c(t)=\rho(t)c$ where 
$\rho(t)=\Big(\frac 92 U(c)\Big)^\frac 13 t^{\frac 23}$.
For any $[a,b]\subset \R^+$, $v\in C^2([a,b],\E)$ such that
$v(a)=v(b)=0$,  $\langle c, v(t) \rangle_m =0$, we consider the variation of $\ga_c$
of the form
\begin{equation}
\ga^\ep(t)=\rho(t)(c+\ep v(t))  
\label{variations}
\end{equation}
so that
\[\frac{d\ga^\ep(t)}{d\ep}=\rho(t)v(t),\, 
\frac{d^2\ga^\ep(t)}{d\ep^2}\Big|_{\ep=0}=0\]
\begin{align*}
  A(\ga^\ep;a,b)&=\frac{1}{2}\int_a^b \langle \dot\ga^\ep(t), \dot\ga^\ep(t) \rangle_m
+\int_a^b U(\ga^\ep(t))\,dt
\end{align*}
\begin{align*}
  \frac{dA(\ga^\ep;a,b)}{d\ep}&=
\int_a^b \langle \dot\ga^\ep(t), \frac{d\dot\ga^\ep(t)}{d\ep}  \rangle_m\,dt+
\int_a^b  \langle \nablam U(\ga^\ep(t)),  \frac{d\ga^\ep(t)}{d\ep} \rangle_m\,dt.\\
\end{align*}
We have
\begin{align*}
  \frac{d^2A(\ga^\ep;a,b)}{d\ep^2}\Big|_{\ep=0}
&=\int_a^b\Big[
\langle \frac{d\dot\ga^\ep(t)}{d\ep} , \frac{d\dot\ga^\ep(t)}{d\ep} \rangle _m+
\langle \frac{d\ga^\ep(t)}{d\ep},  D\nablam U(\ga_c(t))\frac{d\ga^\ep(t)}{d\ep} \rangle_m \Big]_{\ep=0} dt\\
\notag =\int_a^b&\Big[\rho(t)^2 \langle \dot v(t),  \dot v(t) \rangle_m+
2\rho(t)\dot\rho(t) \langle v(t), \dot v(t) \rangle_m +\dot\rho(t)^2 \langle  v(t),  v(t) \rangle_m \Big]dt\\
&+\int_a^b\rho(t)^{-1} \langle  v(t),  D\nablam U(c)v(t) \rangle_m \,dt
\end{align*}
where in the last line we used that $D\nablam U$ is homogeneous of degree $-3$. 
We remark that the quantity 
 $\langle  v,  D\nablam U(c)v  \rangle_m$ occuring in the last term is
just the Hessian of $U$ at $c$, evaluated at the vector $v$.  
See the remark at the end of this subsection.
Integrate by parts and use that $\ddot \rho = -U(c)/\rho^2$ to get 
\begin{align*}
  \int_a^b\Big[ 2\rho(t)\dot\rho(t) \langle v(t)\dot v(t) \rangle_m+ \dot\rho(t)^2 \langle   v(t),   v(t) \rangle_m \Big]dt 
  &=\int_a^b\dot\rho(t) \frac{d}{dt}(\rho(t) \langle v(t), v(t) \rangle_m)dt\\
=-\int_a^b\ddot\rho(t)\rho(t) \langle v(t), v(t) \rangle_m dt
&=\int_a^b \rho(t)^{-1}U(c) \langle v(t), v(t) \rangle_mdt,
\end{align*}
so that
\begin{align}
\label{hessA}
  \frac{d^2A(\ga^\ep;a,b)}{d\ep^2}\Big|_{\ep=0}
&=\int_a^b\rho(t)^2  \langle\dot v(t), \dot v(t)\rangle_m \,dt
+\int_a^b
\rho(t)^{-1}[ \langle v(t), D\tilde\nabla U(c)v(t) \rangle_m ]\,dt.
\end{align}
where $D\tilde\nabla U(c) = D\nabla U(c) + U(c)I$.  
This is exactly the quantity (\ref{eq:dgradtilde}) which occurred in the computation of the eigenvalues in the last section.

 Now recall from proposition~\ref{prop_Eulereevals} that $D\tilde\nabla U(c)$ has an eigenvector, say $\delta s =z$ with a negative eigenvalue
 $\a_1= -U(c)\nu$.   Take  $v(t)=\fui(t)z$ for the variation of equation (\ref{variations})
where $\fui\in C^2([a,b])$ with $\fui(a)=\fui(b)=0$. Plugging into equation~(\ref{hessA})  we find that
\begin{equation}
\label{eq:Q}\frac{d^2A(\ga^\ep;a,b)}{d\ep^2}\Big|_{\ep=0} = Q(\fui;a,b)=
\int_a^b\Big[\rho(t)^2\dot\fui(t)^2+\a_1\rho(t)^{-1}\fui(t)^2\Big]\,dt.
\end{equation}
This quadratic form in $\fui$ is positive definite i.e. $Q(\fui;a,b)\ge 0$, 
for any $[a,b]\subset (0,\infty)$ if and only if its   Euler-Lagrange equation
\[ (\rho(t)^2 y')'-\a_1 \rho(t)^{-1}y=0\]  
is disconjugate on $(0,\infty)$ (\cite{Hartman} Section XI.6).
Plug in the expressions for $\rho$ and $\a_1$ to  find that this Euler-Lagrange equation reads
\begin{equation}\label{eulerdf}
 t^2y''+\frac 43 ty'+\frac{2}{9}\nu y=0 
\end{equation}
which has   solutions $t^r$ where   $r$ a root of the indicial equation 
$r^2+\dfrac 13 r+\dfrac{2}{9}\nu = 0$.    Equation \eqref{eulerdf}
fails to be disconjugate
if and only if $r$ has an imaginary part,  which is to say
iff and only if the discriminant, $\Delta = \dfrac{1-8\nu}{9}$ is negative.  
$\Delta$  is negative if and only if 
$\nu>\fr18$
  in which case  the solutions of \eqref{eulerdf} are
  \[y(t)=At^{-1/6}\cos(a\ln t)+Bt^{-1/6}\sin(a\ln t);\quad
a^2= \frac{1}{4} |\Delta| = \frac{1}{36}(8\nu-1 ).\]
which   has in fact infinitely many conjugate points on $(0, \infty)$. 

To finish the proof of part (A) just note that our instability condition $\nu>\fr18$
is precisely the condition for spiraling at the Eulerian restpoint.

\end{proof}

{\sc Remark on Hessians.} Some words are in order regarding the term $D\tilde\nabla U$
occurring in the formula and its relation to the Hessian $D^2 U$.
 If $f$ is any smooth function on a real vector space $\E$
then its Hessian $D^2 f(p)$ at $p \in \E$ is the  coordinate independent bilinear symmetric form
defined by $\frac{d^2}{d \e^2} f(p + \epsilon v)|_{\epsilon=0} = (D^2 f (p)) (v,v)$. If $\langle \cdot, \cdot \rangle$
is any inner product on $\E$, then $D^2 f  (p) (v,v) = \langle D \nabla f (p) (v), v \rangle$ where
$\nabla f$ is the gradient of $f$ with respect to the inner product, whereas   the derivative $D$ of $D \nabla f$ is the usual 
derivative, or Jacobian, of  vector fields $X$ on a vector space, given by  $D X (p) (v) = \frac{d}{d \epsilon}|_{\epsilon=0} X(p+ \epsilon v)$,
and  it is independent of the  inner product.

\subsection{Proof of Part (B) of Lemma \ref{main-lemma}.}

\begin{proof} 

We now consider a parabolic motion asymptotic to an Euler central configuration $c$.
Shifting the origin of time, if necessary, such a solution can be written as  $\a_0(t)=\rho(t)(c+\be(t))$, $t\ge t_0>0$,
where $\rho(t) c$ is the parabolic homothetic solution and $\beta(t)=O(t^{-d})$ 
as $t \to \infty$ for some $d>0$.  

A  bit of explanation is in order regarding $\beta$'s   rate of convergence to zero.  
If we  represent  $\a_0$ in McGehee (spherical) coordinates we get 
  $\a_0 (t) = r(t) s(t)$ with $s(t)$ normalized, $s(t) \to c$ as $t \to \infty$. We are also supposed to parameterize the curve  using $\tau$ instead of $t$. 
The stable manifold theorem applied
to the equilibrium point for $c$ yields $|s(\tau) -c| \le A e^{-\mu \tau}$ for $\tau$ sufficiently large,  where $\mu >0 $ is any number such that $-\mu$ is  greater than
all of  the negative eigenvalues from proposition~\ref{prop_evals} applied to $c$.  
 Our  two  representations of $\a_0$ are related by  $r(t) = \rho(t) \sqrt{ 1 + |\beta(t)|^2}$,
$s(t) = (c + \beta(t))/ \sqrt{ 1 + |\beta(t)|^2}$. 
  Integrating the relation
$r^{3/2}d \tau =dt$ we have that $t = A e^{3 v \tau/2} +\eta$
where $v=\sqrt{2U(c)}$ and where $\eta$ is exponentially small relative to the first term. Since $e^{-\mu \tau} \sim t^{-\frac{2}{3} \frac{\mu}{v}}$  it follows that $\beta (t) \to 0$
at a rate $t^{-d}$ with $d = 2 \mu/3 v$.  

Now take a variation $\a_\ep(t)=\rho(t)(c+\be(t)+\ep v(t))$ 
with $v(t)=\fui(t)z$ as before.
So \[\frac{d\a_\ep(t)}{d\ep}=\rho(t)v(t),\, 
\frac{d^2\a_\ep(t)}{d\ep^2}\Big|_{\ep=0}=0\]
\begin{align*}
\frac{d^2A(\a_\ep;a,b)}{d\ep^2}\Big|_{\ep=0}
&=\int_a^b\rho(t)^2\langle\dot v(t),\dot v(t)\rangle_m,dt\\
&+\int_a^b\rho(t)^{-1}
[U(c) \langle v(t),v(t)\rangle_m+D^2U(c+\be(t))(v(t), v(t))]\,dt\\
&=Q(\fui;a,b)+ (D^2U(c+\be(t))-D^2U(c))(z,z)\fui(t)^2
\end{align*}
with $Q$ as per equation~(\ref{eq:Q}).  (We have written the Hessian of $U$  in the form $D^2 U$ as per the remark above on Hessians,
rather than in the $D \nabla_m U$ form.)  According to Part (A) there are $[a,b]\subset (0, \infty)$, $\fui_1\in C^2([a,b])$ such that 
$\fui_1(a)=\fui_1(b)=0$ and $Q(\fui_1;a,b)<0$.

Defining $\fui_\l(t)=\fui_1(\frac t\l)$ we have 
$\dot\fui_\l(t)=\l^{-2}\dot\fui_1(\frac t\l)$
\[Q(\fui_\l;\l a,\l b)=\int_{\l a}^{\l b} \l^{-\frac 23}
\Big[\rho(\tfrac t\l)^2 \dot\fui_1(\tfrac t\l)^2-
\mu\rho(\tfrac t\l)^{-1}\fui_1(\tfrac t\l)^2\Big]dt
=\l^\frac 13 Q(\fui_1;a,b)\]
For $\l$ sufficiently large and $t\ge\l a$ we have that
$|\be(t)|\le C_1 t^{-d}$  and  so
$\|D^2U(c+\be(t))-D^2U(c)\|\le C_2 t^{-d}$. Thus
\begin{align*}
  \int_{\l a}^{\l b}\|D^2U(c+\be(t))-D^2U(c)\|\frac{\fui_\l(t)^2}{\rho(t)}dt\le
C_2\int_{\l a}^{\l b} \frac{\fui_\l(t)^2}{\rho(t)t^d}dt=C_2\l^{\frac 13-d}\int_a^b\frac{\fui_1(s)^2}{\rho(s)s^d}ds.
\end{align*}
Using $v(t)=\fui_\l z$ we have
\[\frac{d^2A(\a_\ep;\l a,\l b)}{d\ep^2}\Big|_{\ep=0}\le \l^\frac 13\left(Q(\fui_1;a,b)+
C_2\l^{-d}\int_a^b\frac{\fui_1(s)^2}{\rho(s)s^d}ds\right)<0.\]
for $\l$ sufficiently large. 
\end{proof}

\section{Symplectic Structure,  Lagrangian submanifolds, and Proofs of lemma ~\ref{Lag_thm} and theorem ~\ref{converse}. }
\label{sec:Lag}
The differential equations of the three-body problem preserve the standard symplectic structure on $\R^{12}$
$$\omega = m_1 dq_1\wedge dv_1 + m_2 dq_2\wedge dv_2+ m_3 dq_3\wedge dv_3.$$
Here, as usual, the wedge of vectors of one forms means adding the componentwise wedges so, for example, 
$(dx,dy)\wedge (du,dv) = dx\wedge du + dy\wedge dv.$
The restriction of the flow to $X =\E\times \E $ preserves the restriction of $\w$.  The pullback of $\w$ under the change of variables
$q_i=r\,s_i, v_i = r^{-\fr12}z_i$ is
$$\Omega_r = \sum_{i} m_i\left(r^{\fr12}\,ds_i\wedge dz_i + r^{-\fr12}dr\wedge s_i\cdot dz_i +\fr12 r^{-\fr12} dr\wedge z_i\cdot ds_i\right).$$
If we use $u = 1/r$ instead we get
$$\Omega_u = \sum_{i} m_i\left(u^{-\fr12}\,ds_i\wedge dz_i + u^{-\fr32} s_i\cdot dz_i \wedge du +\fr12 u^{-\fr32}  z_i\cdot ds_i\wedge du\right).$$
In both cases we restrict to the eight dimensional subset $X$ of $\R^{13}$ where $r>0, u>0$ and where the normalizations (\ref{eq:normalize}) hold.
\begin{lemma}\label{lemma_symplectic}
Let  $p(\tau)$ be any solution of the blown-up differential equations with $r(\tau)>0, u(\tau)>0$ and let
 vectorfields $a(\tau), b(\tau)$ be solutions of the variational equations along $p(\tau)$ which are tangent to an energy manifold.  Then $\Omega_r(p(\tau))(a(\tau),b(\tau))$ and
 $\Omega_u(p(\tau))(a(\tau),b(\tau))$ are constant.
\end{lemma}
\begin{proof}
Let $\xi$ denote the vectorfield on $\R^{13}$ given by \eqref{eq:mcgeheer},
\eqref{eq:mcgehees}, \eqref{eq:mcgeheez}.  Let $\eta = r^{-\fr23}\xi$ be the same vectorfield without the change of timescale.  Since $\eta$ is the pullback of the Hamiltonian field, it preserves the pullback form $\Omega_r$.  In other words, the Lie derivative
$$L_\eta\Omega_r = \fr{d}{dt}\phi_t^* \Omega_r|_{t=0} = 0.$$
Since $\xi=  f \eta$, where $f=r^{\fr32}$, Cartan's formula gives
$$L_\xi \Omega_r = d(\iota_\xi \Omega_r) + \iota_\xi d\Omega_r = d(f\,\iota_\eta \Omega_r)+0=d(fdH)= df\wedge dH.$$
Here we used the fact that $\iota_\eta \Omega_r = dH$ which is the pullback of the differential form version of Hamilton's equations.  

If $p,a, b$ are as in the statement of the lemma then
$$\fr{d}{d\tau}\Omega_r(p)(a,b) = L_\xi \Omega_r(p)(a,b) =(df\wedge dH)(p)(a,b)=0$$
since $dH(p)(a))=dH(p)(b)=0$.
\end{proof}

Using this lemma we can prove Lemma~\ref{Lag_thm}

\begin{prop}\label{prop_LagrangeLagrange}
Let $l=(u,s,z) = (0, s, v s)$ be one of the Lagrange restpoints at infinity with $v>0$ and let $W^s_+(l)$ denote the part of the stable manifold with $u>0$.  Then $W^s_+(l)$ is a four-dimensional invariant manifold which is a Lagrangian submanifold of $X$.  Similarly, at the restpoints with $v<0$ the unstable manifold $W^u_+(l)$ is Lagrangian.
\end{prop}
\begin{proof}
Propositions~\ref{prop_evals} and \ref{prop_Lagrangeevals} give the eight eigenvalues of the variational equations at $l$.  For positive masses, the quantity $k$ from proposition~\ref{prop_Lagrangeevals} satisfies $0\le k <1$.  It follows that if $v>0$ then there are three positive eigenvalues 
$$v, \fr{-v}{4}\left( 1 - \sqrt{13\pm 12\sqrt{k}}\right)$$
four negative eigenvalues
$$-v, -v, \fr{-v}{4}\left( 1 +\sqrt{13\pm 12\sqrt{k}}\right)$$
and one zero eigenvalue.  The latter is due to the rotational symmetry.  In fact $l$ is part of a circle of equilibria.  This circle is normally hyperbolic
(\cite{Hirsch}, p.1)  so each equilibrium has a four-dimensional stable manifold.

The first negative eigenvalue $-v$ has eigenvector $(\d u,\d s, \d z)= (1,0,0)$ and it follows that the stable manifold has an open subset $W^s_+(l)$ with $u>0$.  Moreover, the other three stable eigenvectors are in the subspace $\delta H=0$.  It follows that $W^s_+(l)$ is contained in the energy manifold $\{H=0\}$.  Using blown-up coordinates we need to show that the two-form $\Omega_u$ vanishes on tangent vectors to $W^s_+(l)$.  Let $a_0, b_0$ be two tangent vectors to $W^s_+(l)$ at a point $p_0 \in W^s_+(l)$.  To show that $\Omega_u(p_0)(a_0,b_0)=0$ it suffices, by 
lemma~\ref{lemma_symplectic}, to show that $\Omega_u(p(\tau))(a(\tau),b(\tau))\rightarrow 0$ as $\tau\into\infty$.  For this we need estimates on the exponential decay of $u$ and the components of $a,b$.

Since $u' = -v u$ and $v(\tau)$ converges exponentially to the value $v$ at the restpoint we have a lower bound
$u(\tau) \ge c \exp(-v\tau)$ for some constant $c>0$ which depends on the particular solution $p(\tau)$ under consideration.
To see this note that
$$u(\tau) \exp( v\tau) = u(0)\exp( \int_0^\tau (v-v(s))\,ds)$$
and the integral is bounded above and below since the integrand tends to $0$ exponentially.  
Since $u(0)>0$ we get a positive lower bound $c$ as required.   The lower bound on $u$ gives 
upper bounds
$$u^{-\fr12} \le c_1\exp(\fr12 v\tau)\qquad u^{-\fr32} \le c_2\exp(\fr32 v\tau)$$
for the coefficients in $\Omega_u$.

A similar argument applies to  the variational differential equations.  We have $\d u' = -v(\tau)\d u$.  Since $v(\tau)\rightarrow v$ exponentially, we have an estimate  of the form $|\d u(\tau)| \le c_3 \exp(- v\tau)$ for every solution of the variational equations.   
Since $p(\tau)\rightarrow l$ exponentailly, the other components $\d s, \d z$ also decay at a rate governed by the eigenvalues at $l$.  The weakest of the attracting eigenvalues is
$$\l_w =\fr{-v}{4}\left( 1 +\sqrt{13- 12\sqrt{k}}\right) < -\fr{v}{2}$$
so we will have upper bounds of the form
$$|\d s_i(\tau)| \le c_4 \exp(\l_w v\tau) \qquad  |\d z_i(\tau)| \le c_4 \exp(\l_w v\tau).$$
Substituting these estimates into the formula for $\Omega_u$ gives
$$\Omega_u(p(\tau))(a(\tau),b(\tau)) \le c_5 \exp((\fr{v}2 + \l_w)\tau) \rightarrow 0.$$
\end{proof}
\begin{proof}[Proof of Lemma~\ref{Lag_thm}]
 From lemma~\ref{lemma_lambda} we see that the stable space of a Lagrange respoint at infinity is generated by the
eigenvector $(1,0,0)$ and 3 eigenvectors $(0,\d s_\a,k_{\a-}\d s_\a)$ for 
eigenvectors $\d s_\a$ of the 3 eigenvalues $\af$ of $D\tilde\nabla U(s)$.
Thus the projection of the stable space is the  whole tangent space
at $(0,c)$ of the configuration space.  
It follows from the implicit function theorem that  $W^s_+(l)$ is a graph near infinity. More precisely, there is a product neighborhood 
$V$ of $(0,c)$ in the blown-up configuration space $[0,\infty) \times S^3$ 
and a smooth map $(u,s) \mapsto y(u,s)$ from $V$ to the space $\E$ of blown up velocities
such that the graph of this map   coincides with the stable manifold
of $l$  in some neighborhood of $l$.  Now being Lagrangian does not make sense at  $u =0$
since the symplectic structure explodes, so in the statement of lemma 2, when we say that $W^s_+(l)$ is a ``Lagrangian graph"
we mean over $V \setminus \{u = 0\}$. 
\end{proof}

\begin{definition}  
\label{def:nbhdInfinity} By a   ``neighborhood of $c$ at infinity'' we mean a  neighborhood of the  form described in the end of the  proof 
immediately above.  
When expressed in   $\E$ such a  neighborhood is a truncated open cone  consisting of   those points  $q \in \E$ of the form $q= r s$ where
$| s | = 1$,  $|s-c| < \delta $ and $u = 1/r < \delta$.  
\end{definition}

\begin{cor}
 \label{Busemann}
 Let  $B_c$ be the Buseman function (definition  \ref{defBuseman},  
proposition \ref{Buseman})   
associated to a homothetic parabolic  Lagrange solution  $\ga_c$.
Then there is a   neighborhood of $c$ at infinity,  $V \subset \E$ (see above definition \ref{def:nbhdInfinity}) 
on which   $B_c$ is smooth
 and such  that    
\[\{(\ga(t),\dot\ga(t)^{*}):\ga\hbox{ a curve calibrated by } B_c, \ga(0)\in V\}
= W^s_+(l)\cap (V\times\E ^* ), \quad l=(0,c,vc).\]
(In this formula   $v^*$ denotes the Legendre transform of  the velocity $v$, which is its dual covector.)  
\end{cor}

The reader may wish to consult Remark 2 following Definition \ref{wkam} for context here. 
\begin{proof}
 $W^s_+(l)$ is a Lagrangian graph near $c$ at infinity, and is the graph over some 
neighborhood $V \subset \E$ of infinity, as per the above  terminology.  
Thus, being Lagrangian  there is   a  differentiable function $f$
defined on  $V$ such that $W^s_+(l) \cap (V\times \E)=\graf df$. 
For any $x\in V$ there is a unique motion $\ga(t)$ with 
$\ga(0)=x$, $(x,\dot\ga(0) ^*)\in W^s_+(l)$ and it is given 
by the solution to $\dot\ga^*=df(\ga)$ with $\ga(0)=x$.
Here $w^*$ denotes the dual of $w$ with respect to the mass metric --
which is to say - the inverse Legendre transform of $w$ relative to our Lagrangian.

On the other hand for any $x\in V$ there is a $\a:[0,\infty)\to\E$ 
that  starts at $x$ and calibrates $B_c$. For $t>0$, $B_c$ is differentiable at $\a(t)$ and 
$dB_c(\a(t))=\dot\a (t) ^*$. Since $\a$ is asymptotic to $c$, we
have that $(\a(t),\dot\a (t)^*)\in W^s_+(l)$ so, 
by the graph description $\dot\a (t) ^*=df (\a(t))$.  Thus  
$dB_c(\a(t))=df(\a(t))$, proving the corollary.
(We also have that   $B_c=f+k $ on   $V$, $k$ a constant.)  
\end{proof}
\begin{proof}[Proof of Theorem~\ref{converse}]
Let $\gamma(t)$ tend parabolically to a Lagrange central configuration
$c$ and let  $V$ be the
neighborhood  of $c$ at infinity as in Corollary~\ref{Busemann} and lemma 2.  
Then, since  $(\ga(t), \dot \ga (t)^* )$
lies on the stable manifold $W = W_+^s(l)$ we must have that 
$\gamma([T, \infty)) \subset V$  for $T$ large enough.
Consequently for $t \ge T$ we have that  $\dot\ga(t)^{*}=dB_c(\ga(t))$. 
The curve $\a:[0,\infty)\to\E$ calibrated by
$B_c$ that starts at $\ga(T)$ is also a solution of the differential equation
$\dot z^*= dB_c(z)$. By uniqueness of solutions we have $\a(t)=\ga(t+T)$.
Then $\ga:[T,\infty )\to\E$ is calibrated by  $B_c$ and in particular
it is a  free time minimizer.
\end{proof}


\section{Appendix}
The appendix provides the proofs of the propositions about eigenvalues used above.

\begin{proof}[Proof of Proposition~\ref{prop_evals}]
The first two eigenvalues in the list are from  $(\d r, \d s, \d z) = (1,0,0)$ and $(0, 0, s)$.  The others come from the eigenvalues of $B$ found in the lemma.  The eigenvalues $-v, 0$  come from the eigenvector $\d s = s^\perp$ with $\a = 0$ and the other two come from the two nontrivial eigenvalues.

For the restpoints at infinity, the computation is the same except that the first eigenvalue on the list is now associated to $(\d u, \d s, \d z) = (1,0,0)$ and has eigenvalue $-v$ instead of $v$.
\end{proof}

To prove propositions~\ref{prop_Lagrangeevals} and \ref{prop_Eulereevals} we need to find the nontrivial eigenvalues $\a_1,\a_2$ of $D\tilde\nabla U(s)$ for the equilateral and collinear central configurations of the three-body problem.  These can be deduced from the work of Siegel but we will give a quick discussion here.  

It is straightforward to calculate the $6\times 6$ matrix $D\nabla U(s)$ with the result
\begin{equation}\label{eq:D2U}
D\nabla U(s)= \m{D_{11}&D_{12}&D_{13}\\
D_{21}&D_{22}&D_{23}\\
D_{31}&D_{32}&D_{23}
}
\end{equation}
where the $2\times 2$ blocks are
$$D_{ij}  = \fr{m_im_j}{r_{ij}^3}\left(I - 3u_{ij}u_{ij}^{t} \right), \quad u_{ij} = \fr{s_i-s_j}{r_{ij}}\qquad \text{ for }i\ne j$$
and 
$$D_{ii} = -\sum_{j\ne i}D_{ij}.$$

It is more convenient to work with the matrix
$$P = \fr{I(s)}{U(s)}M^{-1}D\nabla U(s).$$
Since $P$ is invariant under scaling and translation, it can be computed without imposing the normalizations ({\ref{eq:normalize}).
If $\b$ is an eigenvalue of $P$ then $\a = U(s)(\b +1)$ is an eigenvalue of $D\tilde\nabla U(s)$ for the corresponding normalized $s$.  So we are reduced to finding the nontrivial eigenvalues $\b_1,\b_2$ of $P$.

\begin{proof}[Proof of Proposition~\ref{prop_Lagrangeevals}]
Consider an equilateral triangle configuration $s$.  Working with $P$ we can use  the unnormalized configuration
$$s_1=(1,0) \qquad s_2=(-\fr12,\fr{\sqrt{3}}2)\qquad s_3=(-\fr12,-\fr{\sqrt{3}}2)$$ 
for which
$$U(s) =  \fr{m_1m_2+m_1m_3+m_2m_3}{\sqrt{3}}\qquad  I(s) = \fr{3(m_1m_2+m_1m_3+m_2m_3)}{m}.$$
Using these together with (\ref{eq:D2U}) gives
\[
P = 	
\frac{1}{4m}	\m{
			5(m_2+m_3)&3\sqrt{3}(m_3-m_2)&-5m_2&3\sqrt{3}m_2&-5m_3&-3\sqrt{3}m_3\\
			3\sqrt{3}(m_3-m_2)&-(m_2+m_3)&3\sqrt{3}m_2&m_2&-3\sqrt{3}m_3&m_3\\
			-5m_1&3\sqrt{3}m_1&5m_1-4m_3&-3\sqrt{3}m_1&4m_3&0\\
			3\sqrt{3}m_1&m_1&-3\sqrt{3}m_1&-m_1+8m_3&0&-8m_3\\
			-5m_1&-3\sqrt{3}m_1&4m_2&0&5m_1-4m_2&3\sqrt{3}m_1\\
			-3\sqrt{3}m_1&m_1&0&-8m_2&3\sqrt{3}m_1&-m_1+8m_2}
\]
One can guess 4 of the 6 eigenvalues of $P$.  If $e_1=(1,0), e_2= (0,1)$ then $(e_1,e_1,e_1)$ and $(e_2,e_2,e_2)$ are eigenvectors with eigenvalue $\b=0$.
Also $s, s^\perp$ are eigenvectors with eigenvalues $\b = 2,-1$ respectively.  Since the trace of $P$ is $2$, the remaining eigenvalues satisfy
$\b_1+\b_2 = 1$.
Alternatively, the numbers $\g_i=\b_i+1$ which we really want, satisfy $\g_1+\g_2 = 3$.
We can also find the product $\g_1\g_2$ as follows.  We have
$$(\trace P)^2-\trace P^2 = (1+\b_1+\b_2)^2 - (5+\b_1^2+\b_2^2) = 2\b_1\b_2-2 = 2\g_1\g_2-6.$$
With some computer assistance, this gives
$$\g_1\g_2= \fr{27(m_1m_2+m_1m_3+m_2m_3)}{4(m_1+m_2+m_3)^2}$$
Solving the quadratic equation $\g^2 -3\g +\g_1\g_2=0$ gives the  eigenvalues $\a_i =U(s)\g_i = v^2 \g_i/2$ of $D\tilde\nabla U(s)$ listed in the proposition.   Then we get the nontrivial eigenvalues $\lambda$ of  the equilibrium from proposition~\ref{prop_evals}.
\end{proof}

\begin{proof}[Proof of proposition~\ref{prop_Eulereevals}]
Consider a normalized collinear central configuration such that $s_i=(x_i,0)\in\R^2$.
Then the unit vectors $u_{ij} = (\pm 1,0)$ so the $2\times 2$ matrices $D_{ij}$ reduce to
$$D_{ij} = \fr{m_im_j}{r_{ij}^3}\m{-2&0 \\ 0&1}\qquad i\ne j.$$
Rearranging the variables as $q=(x_1,x_2, x_3, y_1,y_2, y_3)$ produces a block structure 
\begin{equation}\label{eq_block}
M^{-1}D\nabla U(s) = \m{2C&0 \\0&-C}
\end{equation}
where $C$ is the $n\times n$ matrix 
$$C= \m{\fr{m_2}{r_{12}^3}+\fr{m_3}{r_{13}^3}&-\fr{m_2}{r_{12}^3}&-\fr{m_3}{r_{13}^3} \\
-\fr{m_1}{r_{12}^3}&\fr{m_1}{r_{12}^3}+\fr{m_3}{r_{23}^3}&-\fr{m_3}{r_{23}^3} \\
-\fr{m_1}{r_{13}^3}&-\fr{m_2}{r_{23}^3} &\fr{m_1}{r_{13}^3}+\fr{m_2}{r_{23}^3}}.$$
An eigenvalue $\mu$ of $C$ determines two eigenvalues 
$$\a = -\mu+U(s), 2\mu +U(s)$$
 for
$D\tilde\nabla U(s) = M^{-1}D\nabla U(s) +U(s)I$.

It is possible to guess two eigenvectors of $C$.  First of all $v_1= (1,1,1)$ is an eigenvector with eigenvalue $0$.  Next, let $v_2= (x_1,x_2,x_3)$ be the vector of $x$-coordinates of the collinear central configuration.  Then it is easy to see that
$$C v_2 = -{M_0}^{-1}\nabla_x U$$
where $\nabla_x$ is the partial gradient with respect to the $x$-coordinates and $M_0=\diag(m_1,m_2,m_3)$.  Since $s$ is a normalized central configuration, we have $C v_2 = U(s) v_2$, so $v_2$ is also an eigenvector, with eigenvalue $U(s)$.   The remaining, nontrivial eigenvalue of $C$ can now be found as
$\mu = \tau -U(s)$ where $\tau=\trace(C)$, i.e., 
$$\tau= \left(\fr{m_1+m_2}{r_{12}^3} +\fr{m_1+m_3}{r_{13}^3} +\fr{m_2+m_3}{r_{23}^3}\right).$$
Therefore the nontrivial eigenvalues of  of $D\tilde\nabla U(s)$ are 
$$\a =  2\tau -U(s), 2U(s)-\tau.$$

To get the form shown in the proposition, let $\nu$ be the translation and scale invariant quantity
$$\nu = \fr{I(s)}{U(s)}\tau -2.$$
Then for the normalized configuration  $\a_1,\a_2 =  -U(s)\nu, U(s)(3 + 2\nu)$ and it remains to show that $\nu$ has the indicated form. 

We just indicate a computer assisted way to prove it.  Using the configuration $s_1=(0,0), s_2= (r,0),s_3=(1+r,0)$ we have
$$r_{12} =r \qquad r_{23} = 1 \qquad r_{13} = 1+r.$$
Substituting these into the formulas for $I(s), U(s), \tau$ expresses $\nu = \fr{I(s)}{U(s)}\tau -2$ as a rational function $\nu(r)$.  Subtracting the expression (\ref{eq_nu}) and factorizing the difference reveals that there is a factor of $g(r)$ in the numerator, where $g(r)$ is the fifth degree polynomial (\ref{eq:g}) giving the location of the central configuration.  So $\nu(r)$ is indeed given by (\ref{eq_nu}) at the central configuration.
\end{proof}

{\bf Acknowledgements}

{\it Montgomery   thankfully acknowledges support by NSF grant DMS-20030177.
Sanchez and Montgomery thankfully acknowledge support of UC-MEXUS grant CN-16-78

\end{document}